\documentclass[leqno]{amsart}
\usepackage[left=3.5cm,right=3.5cm,top=1.5cm,bottom=1.5cm,includeheadfoot]{geometry}
\usepackage{amsmath, amsthm, amssymb, amstext}
\usepackage{color}	
\usepackage{graphicx}
\usepackage{subfigure}
\usepackage[pagewise]{lineno}
\newcommand{\Rr}{\mathbb{R}}

\newcommand{\Ff}{\mathcal{F}}
\newcommand{\Gg}{\mathcal{G}}

\newcommand{\PC}{\mbox{int}(\mathcal{P})}
\newcommand{\Cuz}{C^1_0}

\newcommand{\Coc}{\mathfrak{C}}

\newcommand{\UB}{\overline{u}}

\newcommand{\VB}{\overline{v}}

\newcommand{\Um}{u_{_{\scriptscriptstyle M}}}
\newcommand{\Vm}{v_{_{\scriptscriptstyle M}}}

\newcommand{\MaxU}{u_{_{\scriptscriptstyle M}}}

\newcommand{\MaxV}{v_{_{\scriptscriptstyle M}}}

\newcommand{\MaxX}{x_{_{\scriptscriptstyle M}}}

\newcommand{\MaxY}{y_{_{\scriptscriptstyle M}}}

\newcommand{\MinA}{a_{_{\scriptscriptstyle L}}}

\newcommand{\Lb}{\overline{\lambda}}

\newcommand{\An}{a_\varepsilon}
\newcommand{\Que}{\left[  }
\newcommand{\Qud}{\right]  }
\newcommand{\Pae}{\hspace{-0.05cm}\left( }
\newcommand{\Pad}{\right) }

\newcommand{\Moe}{\left| }
\newcommand{\Mod}{\right| }
\newcommand{\Noe}{\left\|  }
\newcommand{\Nod}{\right\|  }
\newcommand{\NodI}{\right\|_{\infty}  }

\newcommand{\IntO}{\displaystyle\int_\Omega}

\newcommand{\Lapla}{\triangle}

\newcommand{\mI}{\leqslant}
\newcommand{\MI}{\geqslant}

\newcommand{\Id}{\mbox{Id}}

\newcommand{\Lu}{\lambda_1}
\newcommand{\Le}{\lambda^*}

\newcommand{\Su}{\sigma_1}

\newcommand{\Proj}{\mbox{Proj}}

\newcommand{\VPu}{\varphi_1}
\newcommand{\Pu}{\psi_1}
\newcommand{\Ker}{\mbox{Ker}}
\newcommand{\Img}{\mbox{Rg}}
\newcommand{\Cod}{\mbox{codim}}
\newcommand{\Dim}{\mbox{dim}}
\newcommand{\Span}{\mbox{span}}

\newcommand{\Vai}{\longrightarrow}
\newcommand{\Chega}{\longmapsto}

\newcommand{\FeOmega}{\Pae\overline{\Omega}\Pad}

\theoremstyle{definition} 

\newtheorem{Theo}{Theorem}[section]

\newtheorem{Prop}{Proposition}[section]
\newtheorem{remark}{Remark}[section]
\title[Nonlocal diffusion elliptic system]{Nonlocal diffusion elliptic system modelling the behaviour of a bacteria and a living nutrient}

\author[Y. B. C. Carranza, M. A. V. Costa, C. Morales-Rodrigo, A. Su\'arez]{M. A. V. Costa, Y. B. C. Carranza,  C. Morales-Rodrigo, A. Su\'arez}

\address[Y. B. C. Carranza, M. A. V. Costa]{Dpto. de Matem\'atica e Computa\c{c}\~ao, Universidade Estadual Paulista - Unesp, Brasil}
\email{yino.cueva@unesp.br, marcos.viana@unesp.br}

\address[C. Morales-Rodrigo, A. Su\'arez]{Dpto. EDAN and IMUS, University of Sevilla, C/ Tarfia s/n, 41013, Seville, Spain}
\email{cristianm@us.es, suarez@us.es}

\subjclass[INFO 1]{35B09, 35B32, 35J47, 92B05}

\begin{document}

\maketitle
\begin{abstract} 
In this paper, we discuss the existence and uniqueness of coexistence states for a class of non-local elliptic system. This problem models the behaviour of a bacteria and a living nutrient, whose diffusion depends on the population of the bacteria in a non-local and nonlinear way. Mainly, we employ bifurcation methods and the Implicit Function Theorem to obtain the existence and uniqueness of positive solution.
\end{abstract}
%%%%%%%%%%%%%%%%%%%%%%%%%%%%%%%%%%%%%%%%%%%%%%%%%%
%%%%%%%%%%%%%%%%%%%%%%%%%%%%%%%%%%%%%%%%%%%%%%%%%%
\section{Introduction}
In this paper, we deal with the existence of coexistence states of the following nonlocal and nonlinear elliptic system
	\begin{align}\tag{P}
		\begin{cases}\begin{array}{rll}
				-a\Pae\IntO v\Pad\Lapla u=&\hspace*{-0.25cm}\lambda u-u^2+buv&\mbox{ in }\Omega	,	\\
				-\Lapla v+\sigma v=&\hspace*{-0.25cm}\rho u &\mbox{ in }\Omega,			\\
				u=v=&\hspace*{-0.25cm}0&\mbox{ on }\partial\Omega ,
		\end{array}\end{cases}
	\end{align}
	where $ \Omega $ is a bounded regular domain of $ \Rr^N $, $ N\geqslant 1 $, $ a:\Rr\to [0,+\infty) $ is a continuous function, $ b\in\Rr $, $\rho\MI 0$, $ \sigma>0 $ and $\lambda\in\Rr$ will be considered as bifurcation parameter.\\
	
A particular system of (P) was introduced in [1] to model the behaviour of a bacteria, with density $v$, located in a container $\Omega$, and $u$ plays the role of the nutrient. Specifically, in [1] the first equation in (P) is
$$
-a\Pae\IntO v\Pad\Lapla u=f(x),\quad \mbox{in $\Omega$,}
$$
where $f$ is a constant rate of the nutrient. 

In our model, this nutrient is another living organism that grows following a logistic law, $\lambda u- u^2$, $\lambda$ is the growth rate, and interacts with the bacteria with rate $b$, in a competitive or cooperative way depending on the sign of $b$, negative in the first case, positive in the second one. In the particular case $b=0$, this interaction does not occur. Finally, the bacteria has a constant $\sigma$ death rate and a source rate $\rho$ depending only of the nutrients. 

Observe that the main novelty of (P) is that the diffusion of the nutrient depends on the population of the bacteria in a nonlocal and nonlinear way, see also [4], [7] and [10] for related works with nonlocal diffusivity term. 

Our main goal in this paper is to give sufficient conditions to assure the existence of a coexistence state, that is, a solution of (P) with both positive components. Let us describe our main results.

In the particular case $\rho=0$, that is, the bacteria does not receive nutrients, then the bacteria dies, $v=0$, and the nutrients follow the classical logistic equation
\begin{align}
		\begin{cases}\begin{array}{rll}
				-a(0)\Delta u=&\hspace*{-0.25cm}\lambda u-u^2&\mbox{ in }\Omega	,		\\
				u=&\hspace*{-0.25cm}0 &\mbox{ on }\partial\Omega,
		\end{array}\end{cases}
\end{align}
%\begin{equation}
%\label{logisintro}
%\left\{
%\begin{array}{ll}
%-a(0)\Delta u=\lambda u-u^2 & \mbox{in $\Omega$,}\\
%u=0 & \mbox{on $\partial\Omega$.}
%\end{array}
%\right.
%\end{equation}
In this case, assuming that $a(0)>0$, it is well-known that (1) possesses a positive solution if and only if $\lambda>a(0)\lambda_1$, where $\lambda_1$ is the principal eigenvalue of the Lapalacian in $\Omega$ under homogeneous boundary conditions. 

Thus, from now on we assume that $\rho>0$. In this case, the results depend on the sign of $b$. We distinguish three cases:
%\begin{enumerate}
%    \item Assume $b=0$. Then, if $\lambda>a(0)\lambda_1$ then (P) possesses at least a coexistence state. Moreover, if $a$ is increasing, then (P) has a unique coexistence state.
 %   \item Assume $b<0$. Then, (P) has at least a coexistence state if $\lambda>a(0)\lambda_1$.
 %   \item Assume $b>0$. Then, (P) has at least a coexistence state if $\lambda>a(0)\lambda_1$ is some of the following conditions holds:
 %   \begin{enumerate}
 %       \item $b$ is small; or
 %       \item 
%        \begin{equation}\label{aa}
%            \lim_{s\to \infty}\frac{a(s)}{s}=\infty.
%        \end{equation}
%    \end{enumerate}
%\end{enumerate}
\begin{itemize}
    \item[(i)] If $b=0$. Then, if $\lambda>a(0)\lambda_1$ then (P) possesses at least a coexistence state. Moreover, if $a$ is increasing, then (P) has a unique coexistence state.
    \item[(ii)]  If $b<0$. Then, (P) has at least a coexistence state if $\lambda>a(0)\lambda_1$ and (P) does not have coexistence state for $\lambda\leq \lambda_1\displaystyle\min_{s\geq 0}a(s) $. 
    \item[(iii)]  If $b>0$. Then, (P) has at least a coexistence state if $\lambda>a(0)\lambda_1$ if some of the following conditions holds:
    \begin{itemize}
        \item[(a)]  $b$ or $\rho$ is small or $\sigma$ is large; or
        \item[(b)]  
        \begin{align}
            \lim_{s\to \infty}\frac{a(s)}{s}=\infty.
        \end{align}
    \end{itemize}
    Moreover, (P) does not possess coexistence state if $\lambda\leq\lambda_0$ for some $\lambda_0\in \Rr$.
\end{itemize}

Although the results of existence in all cases are quite similar, their achievements differ in each case. In particular, we will use the bifurcation technique, and the major difference in the different cases is in the way to get the a priori bounds. Let us emphasize that in the cooperative case ($b>0$), we have obtained the a priori bounds in two different ways. In case (a) we use mainly arguments based on the principle of maximum, while in case (b) we argue by contradiction and and make use of the fact that the diffusion coefficient grows very fast.

 In all the cases, we prove the existence of an unbounded continuum $\mathcal{C}\subset \Rr\times C_0^1(\overline\Omega)\times C_0^1(\overline\Omega)$ of positive solutions of (P). Specifically, we prove:
 
 \begin{Theo}
 \label{teointrobif}
 Assume $a(0)>0$. From the trivial solution $(u,v)=(0,0)$ emanates an unbounded  continuum $\mathcal{C}\subset \Rr\times C_0^1(\overline\Omega)\times C_0^1(\overline\Omega)$ of positive solutions of (P) at
 $$
 \lambda=a(0)\lambda_1.
 $$
 Moreover, if $b\leq 0$ or $b>0$ or $\rho$ is small or $\sigma$, or $b>0$  and $a$ verifies (2), then
 $$
 (a(0)\lambda_1,\infty)\subset \Proj_\Rr(\mathcal{C})\subset (\lambda_0,\infty),
 $$
for  some $\lambda_0\leq 0$. where $\Proj_\Rr(\lambda,u,v)=\lambda$ for $(\lambda,u,v)\in\mathcal{C}$.
 
 As a consequence, there exists at least a positive solution for $\lambda>a(0)\lambda_1$.
 \end{Theo}
 
 In Section 4 we also study the local bifurcation, including the bifurcation direction. This direction depends on the relative size of the coefficients of (P) and $a'(0)$. Specifically, if
$$ a'(0)>\dfrac{\Pae b\rho-\lambda_1-\sigma\Pad \|\varphi_1\|_3^3}{\Lu \rho\|\varphi_1\|_1},
$$
then the direction is supercritical, while if
  $$ a'(0)<\dfrac{\Pae b\rho-\lambda_1-\sigma\Pad \|\varphi_1\|_3^3}{\Lu \rho\|\varphi_1\|_1}
  $$
the direction is subcritical, here, $\varphi_1$ is a positive eigenfunction associated to $\lambda_1$. In Figure 1 we have illustrated two possible bifurcation diagrams.

Moreover, we prove an uniqueness result. We show that when $a$ is increasing, there exists a unique coexistence state of (P) for $b\in (-b_0,b_0)$ for some $b_0>0$. Observe that this uniqueness is optimal in the following sense: if $a$ is increasing  and $b$ is large, the bifurcation direction is subcritical, and then the multiplicity of positive occurs for $\lambda\in(\lambda_1 a(0)-\delta, \lambda_1 a(0))$ for some $\delta>0$ and small.   
 
Finally, we analyze the case $a(0)=0$. In this case we can not apply directly the bifurcation method, but we can use a compactness argument to show the existence of positive solution of (P) for all $\lambda>0$.

 %\begin{figure}
 %   \centering
 %   \includegraphics[width=0.5\linewidth]{G11.png}
 %   \caption{Enter Caption}
 %   \label{fig:enter-label}
%\end{figure}

%\begin{figure}
 %   \centering
 %   \includegraphics[width=0.5\linewidth]{G22.png}
 %   \caption{Enter Caption}
 %   \label{fig:enter-label}
%\end{figure}

\begin{figure}[htbp]
    \centering
    \subfigure[Supercritial bifurcation]{\includegraphics[width=0.47\textwidth]{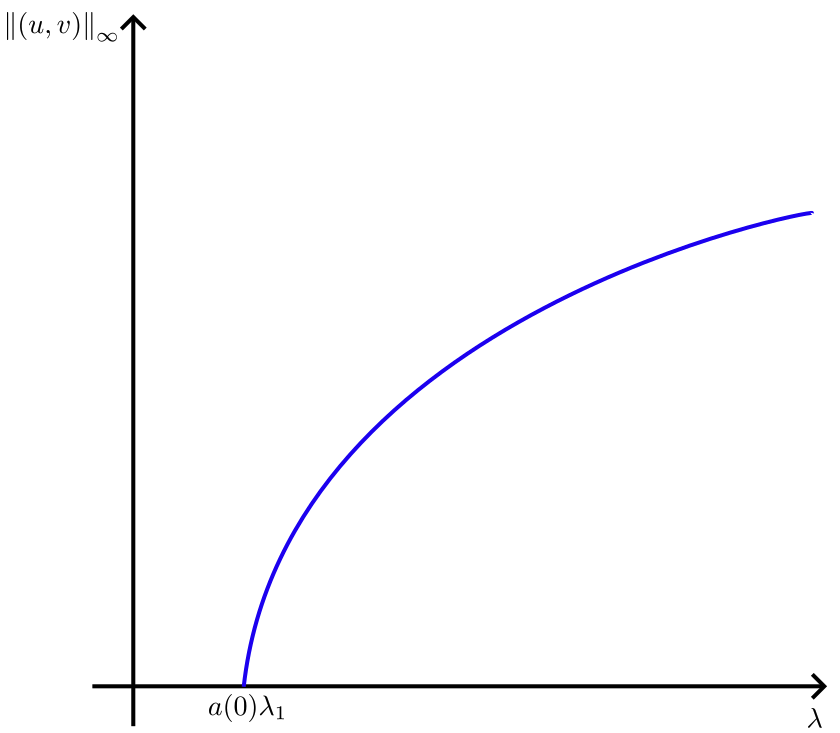}}
\quad
    \subfigure[Subcritial bifurcation]{\includegraphics[width=0.47\textwidth]{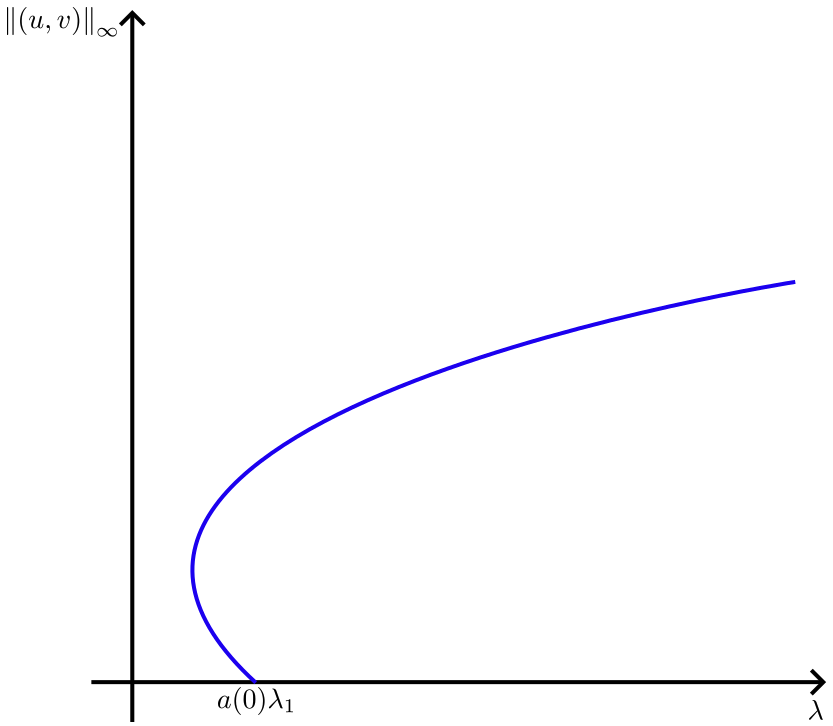}} 
    \caption{Bifurcation diagrams of (P)}
\end{figure}
%ESTO YO NO LO PONDRIA.
 
 %however using the Whyburn's Lemma we are able to prove the existence of an unbounded continuum of positive solutions bifurcating from the trivial solution at $\lambda=0$. For that, we need to impose some conditions on function $a$. Specifically:
%\begin{Theo}
 %\label{teointrobifdege}
 %Assume $a(0)=0$ and:
 %\begin{enumerate}
  %   \item If $b=0$ and $a(s)\geq Cs$ for $s$ small; 
   %  \item If $b\neq 0$ and $a(s)\geq Cs^{1/3}$ for %$s$ small and $b$ verifies $(Hb)$ as $b>0$;
 %\end{enumerate}
 %then there exists an unbounded $\mathcal{C}$ %bifurcating from the trivial solution at %$\lambda=0$.
 %\end{Theo}
 An outline of the work is as follows: in Sections 2 and 3 we give some preliminaries, introduce some notations and prove some a priori bounds of positive solutions of (P) and non-existence results. Sections 4 and 5 are devoted to local and global bifurcation. The uniqueness results are proved in Section 6. Finally, in Section 7 we study the degenerate case $a(0)=0$. 
\section{Preliminaries and Notations}
 %Problems like \eqref{p1} admit three types of solutions: the trivial $\Pae 0,0\Pad$; the semi-trivial $\Pae u,0\Pad$ and $\Pae 0,v\Pad$; and the coexistence state $\Pae u,v\Pad$ when both components are non-trivial and non-negative. In our case, we want to study coexistence state $\Pae u,v\Pad$ for the problem such that $u,v\in\Cuz\FeOmega$.\\
 In this section, we will summarize some definitions and notations that will be used throughout the work.

Given $f\in C\FeOmega$, we denote by
$$
    f_M:=\max\limits_{x\in\overline{\Omega}}f(x)\qquad\mbox{and}\qquad         f_L:=\min\limits_{x\in\overline{\Omega}}f(x).
$$

Moreover, for $d>0$ and $b\in L^\infty(\Omega)$, we denote by $\sigma_1[-d\Delta+b]$ the principal eigenvalue of
%Let $a,b,c\in\Cuz\FeOmega$, with $a\MI k>0$ and $c\MI 0$ in $\Omega$. We %denote by $\Su\Que -a(x)\Lapla+b(x);c(x)\Qud$ the principal eigenvalue of %the problem
\begin{equation}
\left\{
\begin{array}{rll}
        -d\Lapla u+b(x)u=&\hspace*{-0.25cm}\lambda u,&\mbox{ in }\Omega			\\
	u=&\hspace*{-0.25cm}0, &\mbox{ on }\partial\Omega
    \end{array}\right.
\end{equation}
It is well known that the map 
$$
(d,b)\in\Rr \times L^\infty(\Omega)\mapsto \sigma_1[-d\Delta+b]\quad\mbox{is continuous and increasing.}
$$
When $d=1$ and $b\equiv 0$ in $\Omega$, we simply write $\lambda_1$ the principal eigenvalue in (3).

%Furthermore, in the case when $a\equiv 1$, $b\equiv 0$ and $c\equiv 1$ we %will denote the principal eigenvalue for the problem associated simply by %$\Lu:=\Su\Que-\Lapla;1\Qud$.\\
Observe that since $\sigma>0$, it follows that	 
\begin{equation}
	\Su\Que -\Lapla+\sigma\Qud =\lambda_1+\sigma>0.
\end{equation}

Along the paper, we denote by $e_\sigma$ the unique positive solution of
    \begin{align}
    \begin{cases}\begin{array}{rll}
				-\Lapla e+\sigma e=&\hspace*{-0.25cm}1,&\mbox{ in }\Omega			\\
				e=&\hspace*{-0.25cm}0, &\mbox{ on }\partial\Omega,
    \end{array}.\end{cases}
    \end{align}
whose existence and uniqueness are guaranteed by (4).

Our main goal in this work will be to investigate the existence of a non-negative solutions $\Pae u,v\Pad$ for the problem (P). In particular, we distinguish three types of solutions:
 \begin{itemize}
     \item Trivial $\Pae 0,0\Pad$.
     \item Semi-trivial $\Pae u,0\Pad$ and $\Pae 0,v\Pad$.
     \item Coexistence state $\Pae u,v\Pad$ when both components are non-trivial and non-negative.
 \end{itemize}
First, we deal with the case $\rho=0$. Observe that $v$ verifies
\begin{equation}
 -\Delta v+\sigma v=0\quad\mbox{in $\Omega$,}\quad v=0\quad\mbox{on $\partial\Omega$.}   
\end{equation}

Hence, by (4) and using  the maximum principle we get that $v\equiv 0$ in $\Omega$. Then, $u$ verifies
\begin{align}
		\begin{cases}\begin{array}{rll}
				-a(0)\Lapla u=&\hspace*{-0.25cm}\lambda u-u^2,&\mbox{ in }\Omega			\\
				u=&\hspace*{-0.25cm}0, &\mbox{ on }\partial\Omega
		\end{array}.\end{cases}
	\end{align}
Therefore, (7) possesses a unique positive solution if, and only if,
$$
\lambda>\sigma_1[-a(0)\Lapla]=a(0)\lambda_1.
$$
%In such case, the positive solution is the unique positive solution. 
In summary, when  $\rho=0$, (P) possesses a positive solution if and only if $\lambda>a(0)\lambda_1$. In such case, the solution is unique.

From now on, we assume throughout the paper that $\rho>0$.
 
 Observe that (P) does not possess semi-trivial solutions. Indeed, if $u=0$ in $\Omega$, then $v$ verifies (6). Hence, $v\equiv 0$ in $\Omega$. 

On the other hand, if $u\equiv 0$ in $\Omega$, from (6), we conclude that $v\equiv 0$ in $\Omega$. 

Finally, observe that if $(u,v)$ is a coexistence state of (P), then both $u$ and $v$ are strictly positive by the strong maximum principle.

 %In our case, we want to study coexistence state $\Pae u,v\Pad$ for the %problem such that $u,v\in\Cuz\FeOmega$.\\

 %Note that, the problem \eqref{p1} does not admit semi-trivial solution. %Indeed, first consider the case $ u\equiv 0 $, that is:
%	\begin{align}\label{aa1}
%		\begin{cases}\begin{array}{rll}
%				-\Lapla v+\sigma v=&\hspace*{-0.25cm}0,&\mbox{ in }\Omega			\\
%				v=&\hspace*{-0.25cm}0, &\mbox{ on }\partial\Omega
%		\end{array}.\end{cases}
%	\end{align}	

%	Then, there exists a unique solution $ v\equiv 0 $ to problem %\eqref{aa1}. 
 %As $ v\equiv 0 $ satisfies the first equation of the problem \eqref{aa1} follow that $ v\equiv0 $ is the only solution.
 %Now, consider the following case with $ v\equiv 0. $ Then, we have
	
%	There are two possibilities: $ \rho=0 $ and $ \rho>0 $. For $ \rho=0 $ follows that $ u\neq 0 $ and, consequently, $(u,0)$ is solution for the problem. Now suppose that $ \rho> 0 $. Note that $u\equiv 0$  satisfies the first equation of \eqref{bb1}. Consequently, there is no semi trivial solution in this case.\\

%%%%%%%%%%%%%%%%%%%%%%%%%%%%%%%%%%%%%%%%%%%%%%%%%%
\section{A Priori Bounds and Non-Existence of Coexistence States}

In this section, we prove results of non-existence and a priori bounds of coexistence states of (P). For this, we consider separately  the cases $b\leq 0$ and $b>0$.

\begin{Prop}
    Assume that $b\leq 0$.
    \begin{itemize}
        \item[(a)] If $\lambda\leq \MinA\Lu$, then (P) does not possess coexistence states.
        \item[(b)] If $(u,v)$ is a coexistence state of $(P)$, then 
        $$
        u\mI\lambda\quad\mbox{and}\quad v\mI\rho\lambda e_\sigma \quad\mbox{in $\Omega$,}    
        $$
        where $e_\sigma$ is defined in (5).
    \end{itemize}
\end{Prop}
\begin{proof}
    (a) Suppose that $(u,v)$ is a coexistence state of (P), then
    \begin{align*}
        \lambda=\Su\Que-a\Pae\IntO v\Pad\Lapla+u-bv\Qud.
    \end{align*}
    Since $b\leq 0$ and $a(s)\geq\MinA$, it follows that
    %\begin{align*}
     %   \lambda>\Qu-a\Pae\IntO v\Pad\Lapla\Qud.
    %\end{align*}
    %Furthermore, as there $a_{_{\scriptscriptstyle %L}}:=\min\limits_{s\in[0,+\infty)}a(s)$ such that
    %\begin{align*}
    %    a\Pae\IntO v\Pad\MI a_{_{\scriptscriptstyle L}}>0
    %\end{align*}
    %follows that
    \begin{align*}
        \lambda>\Su\Que-\MinA\Lapla\Qud=\MinA\Lu.
    \end{align*}
   % which is a contradiction with $\lambda<0$. Therefore %\eqref{p1} does not possess coexistence states when %$\lambda<0$.
    (b) Let $\MaxX\in\Omega$ such that $\MaxU=u\Pae\MaxX\Pad=\max\limits_{x\in\overline\Omega}u(x)$. Then,
    \begin{align*}
        -a\Pae\IntO v\Pad\Lapla u(\MaxX)=\lambda\Um-\Um^2+b\Um v(\MaxX)\MI 0
    \end{align*}
    and, as a consequence, we arrive at
    \begin{equation}
      \Um\leq \lambda+bv(\MaxX).
    \end{equation}
    Since $b\leq 0$, it follows that
    \begin{align*}
        u\mI\lambda\quad\mbox{in $\Omega$.}
    \end{align*}
    Therefore, from the second equation of (P), it follows that
    \begin{align*}
        -\Lapla v+\sigma v=\rho u\mI\rho\lambda
    \end{align*}
    which yields 
    \begin{align*}
        v\mI\rho\lambda e_\sigma,
    \end{align*}
    This completes the proof.
\end{proof}
Now, we study the case $b>0$.
\begin{Prop}
    Suppose that $b>0$.
    \begin{enumerate}
        \item[(a)] Assume that
    $$
     b\rho\|e_\sigma\|_\infty<1.
    \eqno(H_1)
    $$
    Then, 
     $$
         \|u\|_\infty\leq\frac{\lambda}{1- b\rho\|e_\sigma\|_\infty}\quad\mbox{and}\quad 
        \|v\|_\infty \leq\frac{\rho\lambda \|e_\sigma\|_\infty}{1- b\rho\|e_\sigma\|_\infty}.
        $$
    
    \item[(b)] If
    $$
    \lim\limits_{s\to\infty}\dfrac{a(s)}{s}=+\infty,
    \eqno(H_2)
    $$
     and $\lambda\in \Lambda\subset\Rr$, $\Lambda$ a compact set, then there exists $K>0$ such that
        $$
        \|(u,v)\|_\infty\leq K\qquad\mbox{for all $\lambda\in\Lambda$.}
        $$
    Moreover, there exists $\lambda_0\in\Rr$ such that (P) does not possess coexistece states for $\lambda\leq \lambda_0$.
    \end{enumerate}
\end{Prop}
\begin{proof}
(a)
%    \begin{align}\label{DC1}
 %       \lambda\Um-\Um^2+b\Um\Vm\MI 0\nonumber
  %      &\If\lambda-\Um+b\Vm\MI 0\\
   %     &\If\Um\mI\lambda+b\Vm.
  %  \end{align}
    Using  the second equation of (P), it follows that
    $$
    -\Lapla v+\sigma v=\rho u\leq \rho \Um,
    $$
%    \begin{align*}
 %       -\Lapla v+\sigma v
  %      &=\rho u\\
   %     &\mI\rho\Um\\
    %    &\mI\rho\Pae\lambda+b\Vm\Pad
    %\end{align*}
which yields 
    $$
    v\leq \rho\Um e_\sigma.
    $$
Hence, 
\begin{equation}
    \Vm\leq \rho\Um \|e_\sigma\|_\infty.
\end{equation} 
Using (8), since $b>0$ we get that 
$$
\Um\leq \lambda +bv_M.
$$
By (9), we conclude the first paragraph.
  %  $v$ is bounded. So, follows that
   % \begin{align*}
    %    v\mI\rho\Pae\lambda+b\Vm\Pad e,
    %\end{align*}
    %where $e$ is the solution of \eqref{EEE1}. Taking the %maximum of $v$ and $e$ in the equation above, follows that
    %\begin{align*}
     %   \Vm\mI\rho\Pae\lambda+b\Vm\Pad e_{_{\scriptscriptstyle M}}
     %   \If\Vm\Pae 1-\rho be_{_{\scriptscriptstyle M}}\Pad\mI\rho\lambda e_{_{\scriptscriptstyle M}}.
    %\end{align*}
    %Therefore, when $1-\rho be_{_{\scriptscriptstyle M}}>0$,
    %\begin{align*}\label{DC2}
    %    \Vm\mI\rho\lambda
    %\end{align*}
    %and, consequently, $\Um$ and $\Vm$ are bounded.

    \noindent (b) If we denote $\MaxY\in\Omega$ such that $v(\MaxY)=\MaxV=\max\limits_{x\in\overline\Omega}v(x)$, we get
    \begin{align*}
        -\Lapla v(\MaxY)+\sigma\Vm=\rho u(\MaxY)\mI\rho\Um,
    \end{align*}
and then
\begin{equation}
\sigma\Vm\mI\rho\Um.
\end{equation}
Consequently, using (8)
    \begin{align}
        \dfrac{\sigma}{\rho}\mI\dfrac{\Um}{\Vm}\mI\dfrac{\lambda}{\Vm}+b.
    \end{align}
    Then, combining (10) and (8) we deduce that
    \begin{align}
        \dfrac{1}{b}\Pae 1-\dfrac{\lambda}{\Um}\Pad\mI\dfrac{\Vm}{\Um}\mI\dfrac{\rho}{\sigma}.
    \end{align}

We argue by contradiction. Suppose that there exist $\Pae\lambda_n\Pad\subset\Rr$  and $\Pae u_n,v_n\Pad\in\Cuz\FeOmega\times\Cuz\FeOmega$, positive solutions of (P) such that such $\lambda_n\to\overline{\lambda}<+\infty$ and $\Noe u_n\NodI+\Noe v_n\NodI\to+\infty$. By (11), if $\Noe u_n\NodI\to+\infty$ then $\Noe v_n\NodI\to+\infty$ and, by (12), if $\Noe v_n\NodI\to+\infty$ then $\Noe u_n\NodI\to+\infty$. 
   Hence, we have that
\begin{equation}
\label{ambos}
\Noe u_n\NodI\to+\infty\quad\mbox{and}\quad \Noe v_n\NodI\to+\infty.
\end{equation}

Consider
$$ 
w_n:=\dfrac{u_n}{\Noe u_n\NodI}\quad\mbox{and}\quad
z_n:=\dfrac{v_n}{\Noe u_n\NodI}.
$$
Note that 
    \begin{equation}
        \Noe w_n\NodI=1\quad \mbox{and}\quad \Noe z_n\Nod_\infty \geq \dfrac{1}{b}\Pae 1-\dfrac{\lambda_n}{\|u_n\|_\infty}\Pad.
    \end{equation}
Using these expressions in the second equation of (P), we have
    \begin{align*}
    \begin{cases}\begin{array}{rll}
				-\Lapla z_n+\sigma z_n=&\hspace*{-0.25cm}\rho w_n,&\mbox{ in }\Omega			\\
				z_n=&\hspace*{-0.25cm}0, &\mbox{ on }\partial\Omega.
    \end{array}
    \end{cases}
    \end{align*}
    Then, by Agmon-Douglis-Nirenberg Theorem (see [6]),
    \begin{align*}
        \Noe z_n\Nod_{W^{2.p}}\mI C\Noe w_n\Nod_p\mI C,\qquad p>1.
    \end{align*}
    Hence, there exists $z^*\in C^1(\overline\Omega)$, $z^*\geq 0$ in $\Omega$ such that
    $$
    z_n\to z^*\quad\mbox{in $C^{1}\FeOmega$.}
    $$
    Observe that from (14), $z^*\neq 0$ in $\Omega$.

    Dividing the equation of $u_n$ by $\|u_n\|_\infty^2$, we get
    %\begin{align*}
     %   \dfrac{-a\Pae\IntO v_n\Pad\Lapla u_n}%{\|u_n\|_\infty^2}=\dfrac{\lambda_nu_n}{\|u_n\|_\infty^2}-%\dfrac{u_n^2}{\|u_n\|_\infty^2}+\dfrac{bu_nv_n}{\|u_n\|_\infty^2}
    %\end{align*}
    %and, consequently,
    \begin{align}
        -\dfrac{a\Pae\IntO v_n\Pad}{\|u_n\|_\infty}\Lapla w_n=\dfrac{\lambda_n}{\|u_n\|_\infty}w_n-w_n^2+bw_nz_n\quad\mbox{in $\Omega$,}\quad w_n=0\quad \mbox{on $\partial\Omega$.}
    \end{align}
    On the left side of (15), we have
    $$
    \dfrac{a\Pae\IntO v_n\Pad}{\|u_n\|_\infty}
        =\dfrac{a\Pae\|u_n\|_\infty\IntO\dfrac{v_n}{\|u_n\|_\infty}\Pad}{\|u_n\|_\infty}\\
        =\dfrac{a\Pae\|u_n\|_\infty\IntO z_n\Pad}{\|u_n\|_\infty}\\
        =\dfrac{a\Pae s_n\Pad}{s_n}\IntO z_n,
    $$
     %  \begin{align*}
   %     \dfrac{a\Pae\IntO v_n\Pad}{\|u_n\|_\infty}
    %    &=\dfrac{a\Pae\|u_n\|_\infty\IntO\dfrac{v_n}{\|u_n\|_\infty}\Pad}%{\|u_n\|_\infty}\\
  %      &=\dfrac{a\Pae\|u_n\|_\infty\IntO z_n\Pad}{\|u_n\|_\infty}\\
%        &=\dfrac{a\Pae s_n\Pad}{s_n}\IntO z_n,
 %   \end{align*}
    where 
    $$
    s_n=\|u_n\|_\infty\IntO z_n.
    $$ 
    Since $z^*\gneq 0$, we deduce that $s_n\to+\infty$. Thus
    \begin{align*}
        -\dfrac{a\Pae s_n\Pad}{s_n}\IntO z_n\ \Lapla w_n=\dfrac{\lambda_n}{\|u_n\|_\infty}w_n-w_n^2+bw_nz_n
    \end{align*}
    and, since the right hand side of the \label{AAA11} equation is bounded by $L^p\Pae\Omega\Pad$ it follows from the expression above 
    \begin{equation}
   \Noe w_n\Nod_{W^{2,p}}\mI\dfrac{C\Noe\dfrac{\lambda_n}{\|u_n\|_\infty}w_n-w_n^2+bw_nz_n\Nod_p}{\dfrac{a\Pae s_n\Pad}{s_n}\IntO z_n}.
     \end{equation}
By $(H_2)$, we conclude by (16) that  $\Noe w_n\Nod_{W^{2,p}}\to 0$, and then $w_n\to 0$ in $C^1\FeOmega$. A contradiction with $\|w\|_\infty=1$.

    Next, we prove that (P) does not have positive solutions for $\lambda\leq \lambda_0$. We argue by contradiction. Suppose that there exists a sequence of coexistence states $(\lambda_n,u_n,v_n)$ of (P) and $\lambda_n\to-\infty$. Observe that
    since 
    $$
    \lambda_n=\sigma_1\left[-a\left(\int_\Omega v_n\right)\Delta+u_n-bv_n\right]
    $$
    then $\|v_n\|_\infty\to \infty$. Indeed, assume that $\|v_n\|_\infty\leq C$, then
    $$
    \lambda_n\geq\MinA\lambda_1-bC,
    $$
    a contradiction. 

    Since $\|v_n\|_\infty\to\infty$, then we obtain that $\|u_n\|_\infty\to \infty$. 

    Observe that by (9), we get
    $$
    1\leq\frac{\lambda_n}{\|u_n\|_\infty}+b\frac{\|v_n\|_\infty}{\|u_n\|_\infty}\leq \frac{\lambda_n}{\|u_n\|_\infty}+b\frac{\rho}{\sigma},
    $$
    and then
    \begin{equation}
    1-b\frac{\rho}{\sigma}\leq \frac{\lambda_n}{\|u_n\|_\infty}\leq 0.
    \end{equation}
    Now, we can argue as before to conclude that
     \begin{align*}
        \Noe w_n\Nod_{W^{2,p}}\mI\dfrac{C\Noe\dfrac{\lambda_n}{\|u_n\|_\infty}w_n-w_n^2+bw_nz_n\Nod_p}{\dfrac{a\Pae s_n\Pad}{s_n}\IntO z_n}.
    \end{align*}
    By (17), we can deduce that
    $$
    \Noe w_n\Nod_{W^{2,p}}\to 0,
    $$
    and we complete the proof.
\end{proof}
\begin{remark}
Observe that the map $\sigma \in [0,+\infty)\mapsto e_\sigma\in C(\overline\Omega)$ is decreasing and
$$
e_\sigma\to 0\quad \mbox{uniformly in $\overline\Omega$ as $\sigma\to\infty$.}
$$
Hence, $(H_1)$ holds if $b$ or $\rho$ is small, or $\sigma$ large. 
\end{remark}

%%%%%%%%%%%%%%%%%%%%%%%%%%%%%%%%%%%%%%%%%%%%%%%%%%
\section{Local Bifurcation Analysis}
In this section, we will re-write (P) so that we can apply local bifurcation results to obtain a curve of non-trivial solutions in a neighborhood of the trivial solution, which is possible by the Crandall-Rabinowitz Theorem, see [3].\\

    %Note that, it is possible rewrite the equations of \eqref{p1} as follows
    %\begin{align*}
	%	\begin{cases}\begin{array}{rll}
	%			-a\Pae\IntO v\Pad\Lapla u-\lambda u+u^2-buv=&\hspace*{-0.25cm}0,&\mbox{ in }\Omega			\\
	%			-\Lapla v+\sigma v-\rho u=&\hspace*{-0.25cm}0, &\mbox{ in }\Omega			\\
	%			u=v=&\hspace*{-0.25cm}0, &\mbox{ on }\partial\Omega
	%	\end{array}.\end{cases}
	%\end{align*}
    %where $u,v\in U$. 
We will assume that $a\in C^2(\Rr)$ and consider $\lambda$ as the main bifurcation parameter. Thus, we define the spaces
\begin{equation}
    E:=C_0^2\FeOmega\times C_0^2\FeOmega,\quad F:=  C\FeOmega\times C\FeOmega
\end{equation}
    and the operator $\Ff:\Rr\times E\to F,$ by
    \begin{align*}
	\mathcal{F}(\lambda,u,v)=\begin{bmatrix}	
		-a\Pae\IntO v\Pad\Lapla u-\lambda u+u^2-bvu		\\
		-\Lapla v+\sigma v-\rho u					
	\end{bmatrix}.
    \end{align*}
    
It is clear that the operator $ \mathcal{F}\in C^2\Pae \Rr\times E, F \Pad.$  Moreover, $(u,v)\in E$ is a non-negative strong solution of (P) if, and only if,
    \begin{align}
	\Ff(\lambda,u,v)=0,
    \end{align}
    for $ \lambda\in\Rr $. Furthermore
    \begin{align}
	\Ff(\lambda,0,0)=0,\quad\mbox{for all $ \lambda\in\Rr $.}
    \end{align}
Its derivative at point $\Pae\lambda,0,0\Pad$ is given by
    \begin{align*}
	\mathcal{F}_{(u,v)}(\lambda,0,0)(\xi,\eta)^t=\begin{bmatrix}	
		-a\Pae 0\Pad\Lapla\xi-\lambda\xi\\
		-\Lapla\eta+\sigma\eta-\rho\xi
	\end{bmatrix}.
    \end{align*}

    The next result ensures the existence and uniqueness of a curve of non-trivial solutions from point $(0,0)$ and an specific bifurcation point.

\begin{Theo}
Assume that $a(0)>0$. Define
  $$\lambda^*:=a(0)\lambda_1,
  $$
and $ Z $ the topological complement of $ \Ker\Que\mathcal{F}_{(u,v)}(\Le,0,0)\Qud $ on $E$, that is:
	\begin{align*}
		E=\Ker\Que\mathcal{F}_{(u,v)}(\Le,0,0)\Qud\oplus Z.
	\end{align*}

	Then, $ \Le $ is a bifurcation point of $ \Ff(\lambda, u,v)=0 $ and the set of non-trivial solutions of $ \mathcal{F}=0 $ in a neighborhood of $(\Le,0,0) $ is a unique cartesian curve of class $ C^1 $ with parametric representation on $ Z $, that is:  there are $ \varepsilon>0 $, $ \rho>0 $ and applications of class $ C^1 $
    	\begin{align*}
		\lambda:&(-\varepsilon,\varepsilon)\Vai\Rr && \varphi:(-\varepsilon,\varepsilon)\Vai Z && \psi:(-\varepsilon,\varepsilon)\Vai Z\\
		&\hspace{0.38cm}s\hspace{0.5cm}\Chega\lambda(s)     			 && \hspace{1.1cm}s\hspace{0.3cm}\Chega\varphi(s)
            && \hspace{1.2cm}s\hspace{0.3cm}\Chega\psi(s)
	\end{align*}
		with $ \lambda(0)=\Le $, $ \varphi(0)=\psi(0)=0 $ and
	\begin{itemize}
		\item[(a)] (Existence of non-trivial solutions) The family $ \lambda=\lambda(s) $, $ u=s(\VPu+\varphi(s)) $ and \break $ v=s(\Pu+\psi(s)) $ is a solution curve (nontrivial if $ s\neq 0 $) for the equation (19) that bifurcates from $ (\Le,0,0) $.
		\item[(b)] (Uniqueness) If $ (\lambda,u,v)\in B_\rho(\Le,0,0) $, a ball in $\Rr\times E$ centered in  $(\Le,0,0)$ and radius $\rho$, is any solution for (20) then there exist $ 0<\Moe s\Mod<\varepsilon $ such that
	\begin{align*}
		(\lambda,u,v)=(\lambda(s),s(\VPu+\varphi(s)),s(\Pu+\psi(s))).
	\end{align*}
	\end{itemize}
		Moreover, if $ a\in C^k(\Rr)$ then the maps $ \lambda(s),u(s),v(s)\in C^{k-1}(-\varepsilon,\varepsilon) $.
	\end{Theo}

\begin{proof}
First, we introduce the notation
$$
\mathcal{L}(\lambda):=\mathcal{F}_{(u,v)}(\lambda,0,0).
$$
We will apply Theorem 1.7 of [3]. For that, consider $\lambda$ as the main bifurcation parameter and  we have to show that $ \mathcal{F} $ verifies:
	\begin{itemize}
		\item[(i)] $ \Ker\Que \mathcal{L}(\lambda^*)\Qud=\Span\{(\VPu,\Pu) \} $,
		\item[(ii)] $\Cod\Que\Img\Pae\mathcal{L}(\lambda^*)\Pad\Qud=1$,
		\item[(iii)] $\mathcal{F}_{\lambda(u,v)}(\Le,0,0)(\VPu,\Pu)^t\notin\Img\Que\mathcal{L}(\lambda^*)\Qud$,
	\end{itemize}
 where $\VPu$  and $\Pu$ are positive functions, which will be detailed below.
    
\noindent (i) We will determine $ \Le $ such that 
$$ \Dim\Que\Ker\Pae \mathcal{L}(\lambda^*)\Pad\Qud=1.
$$
Note that, $ \Pae\xi,\eta\Pad\in\Ker\Que \mathcal{L}(\lambda^*)\Qud$ if, and only if, $ \Pae\xi,\eta\Pad $ is solution for the problem
    \begin{align*}
		\begin{cases}\begin{array}{rll}
				-a(0)\Lapla\xi-\lambda^*\xi=&\hspace*{-0.25cm}0 &\mbox{ in }\Omega,			\\
				-\Lapla\eta+\sigma\eta-\rho\xi=&\hspace*{-0.25cm}0 &\mbox{ in }\Omega,			\\
				\xi=\eta=&\hspace*{-0.25cm}0 &\mbox{ on }\partial\Omega.
		\end{array}
  \end{cases}
	\end{align*}
For the first equation
 %   \begin{align*}
  %  \begin{cases}\begin{array}{rll}
	%			-a(0)\Lapla\xi-\lambda\xi=&\hspace*{-0.25cm}0,&\mbox{ in }\Omega			\\
	%			\xi=&\hspace*{-0.25cm}0, &\mbox{ on }\partial\Omega
	%	\end{array}\end{cases}
	%\end{align*}
we deduce that $ \Le:=a(0)\Lu$. Thus, $ \xi\in\Span\{\VPu\} $, where $ \VPu $ is the principal eigenfunction associated to $ \Lu $. For the linear equation 
    \begin{align}\label{EPU}
		\begin{cases}\begin{array}{rll}	
				-\Lapla\eta+\sigma\eta=&\hspace*{-0.25cm}\rho\varphi_1 &\mbox{ in }\Omega,			\\
				\eta=&\hspace*{-0.25cm}0 &\mbox{ on }\partial\Omega,
		\end{array}\end{cases}
	\end{align}
there exists  the unique solution denoted by $\psi_1$. In fact, observe that
\begin{equation}
   \psi_1=K\varphi_1,\qquad K=\frac{\rho}{\lambda_1+\sigma}. 
\end{equation}
Then
$$
	\Ker\Que \mathcal{L}(\lambda^*)\Qud=\Span\{\Pae\VPu,\psi_1\Pad\}.
$$

\noindent (ii) We claim that
$$\Img(\mathcal{L}(\lambda^*))=\Img(-\Delta-\lambda_1 I)\times C_0^1(\overline\Omega).
$$
Given $ (\varphi,\psi)\in\Img(\mathcal{L}(\lambda^*))$, there exists $ (u,v)\in E$ such that
    \begin{align}
		\begin{cases}\begin{array}{rll}
				-a\Pae 0\Pad\Lapla u -a(0)\Lu u=&\hspace*{-0.25cm}\varphi &\mbox{ in }\Omega,			\\
				-\Lapla v+\sigma v-\rho u=&\hspace*{-0.25cm}\psi &\mbox{ in }\Omega,			\\
				u=v=&\hspace*{-0.25cm}0 &\mbox{ on }\partial\Omega.
		\end{array}\end{cases}
	\end{align}
Hence, $\varphi\in \Img(-\Delta-\lambda_1 I)$, that is $\varphi$ verifies
\begin{equation}
\int_\Omega \varphi\varphi_1=0.
\end{equation}
Since $\sigma_1[-\Delta+\sigma]>0$, it is obvious that for any $\psi\in C_0^1(\overline\Omega)$, there exists $v\in C_0^2(\overline\Omega)$ solution of the second equation of (23).

%which is equivalent to	
 %   \begin{align*}
%		\begin{cases}\begin{array}{rll}
%				a(0)u-\Pae-\Lapla\Pad^{-1}\Pae a(0)\Lu %u\Pad=&\hspace*{-0.25cm}\Pae-%\Lapla\Pad^{-1}\Pae\varphi\Pad,&\mbox{ in }\Omega		%	\\
%				v+\Pae-\Lapla\Pad^{-1}\Pae\sigma v-\rho %u\Pad=&\hspace*{-0.25cm}\Pae-\Lapla\Pad^{-1}\Pae\psi\Pad, %&\mbox{ in }\Omega			\\
%				u=v=&\hspace*{-0.25cm}0, &\mbox{ on %}\partial\Omega
%		\end{array}\end{cases}
%	\end{align*}	
%This system can be written in the following matrix form
%\begin{align*}
%	\begin{bmatrix}	
%		a(0)-\Pae-\Lapla\Pad^{-1}\Pae a(0)\Lu\Pad & 0  \\
%		-\Pae-\Lapla\Pad^{-1}(\rho) & \Id+\Pae-%\Lapla\Pad^{-1}\Pae\sigma\Pad
%	\end{bmatrix}\begin{bmatrix}	
%		u\\ v
%	\end{bmatrix}=\begin{bmatrix}	
%		\Pae-\Lapla\Pad^{-1}\Pae \varphi\Pad\\ \Pae-%\Lapla\Pad^{-1}\Pae\psi\Pad
%	\end{bmatrix}.
%\end{align*}
%Thus, $ (\varphi,\psi)\in\Img(\Tt_1) $, where $ %\Tt_1:L^p(\Omega)\times L^p(\Omega)\Vai L^p(\Omega)\times %L^p(\Omega) $ is the following application
%\begin{align*}
%	\Tt_1=\begin{bmatrix}	
%		a(0)-\Pae-\Lapla\Pad^{-1}\Pae a(0)\Lu\Pad & 0  \\
%		-\Pae-\Lapla\Pad^{-1}(\rho) & \Id+\Pae-%\Lapla\Pad^{-1}\Pae\sigma\Pad
%	\end{bmatrix}.
%\end{align*}
%Note that, $ \Tt_1 $ is a compact perturbation of the %identity and, consequently, $ \Tt_1 $ is a Fredholm operator of index $0$. Consequently $ %\Cod\Que\Img\Pae\Tt_1\Pad\Qud=\Dim\Que\Ker\Pae\Tt_1\Pad\Qu%d. $
%Analogously show that \break $ %\Ker\Pae\Tt_1\Pad=\Ker\Que\mathcal{F}_{(u,v)}\Pae %a(0)\Lu,0,0\Pad\Qud $. 
This proves the claim and then
\begin{align}\label{IdP1}
	\Cod\Que\Img(\mathcal{L}(\lambda^*))\Qud=1.
\end{align}

\noindent (iii) Note that
\begin{align*}
    \mathcal{F}_{\lambda(u,v)}(a(0)\Lu,0,0)(\VPu,\Pu)^t	
    =\begin{bmatrix}	
		-\Id & 0 \\
		0 & 0
	\end{bmatrix}\begin{bmatrix}	
		\VPu\\
		\Pu
	\end{bmatrix}=\begin{bmatrix}	
		-\VPu\\
		0
	\end{bmatrix}
\end{align*}
%and, consequently,
%\begin{align*}
%	\begin{bmatrix}	
%		-\Id & 0 \\
%		0 & 0
%	\end{bmatrix}\begin{bmatrix}	
%		\VPu\\
%		\Pu
%	\end{bmatrix}=\begin{bmatrix}	
%		-\VPu\\
%		0
%	\end{bmatrix}
%\end{align*}
%and
%\begin{align*}
%	\Ff_{(u,v)}\Pae a(0)\Lu,0,0\Pad=\begin{bmatrix}	
%		\Id & 0 \\
%		0 & \Id
%	\end{bmatrix}-\begin{bmatrix}	
%		-a\Pae 0\Pad\Lapla -\lambda+\Id & 0		\\
%		-\rho & -\Lapla+\sigma+\Id					
%	\end{bmatrix}
%\end{align*}
%Using Fredholm Alternative follows that $ 	\Img[ \Id-T] = %\Ker[ \Id-T]^\perp $, where
%\begin{align*}
%	T=\begin{bmatrix}	
%		-a\Pae 0\Pad\Lapla -\lambda+\Id & 0		\\
%		-\rho & -\Lapla+\sigma+\Id					
%	\end{bmatrix}
%\end{align*}
%Since $ \Ker\Que\mathcal{F}_{(u,v)}%%(a(0)\Lu,0,0)\Qud=\Span\{\Pae\VPu,\Pu\Pad\} $,  if $ (\varphi,\psi)\in\Img\Que\mathcal{F}_{(u,v)}(a(0)\Lu,0,0)\Qud  $, follows that
%\begin{align*}
%	\int_\Omega \varphi\VPu=\int_\Omega\psi\Pu=0.
%\end{align*}
Thus $\mathcal{F}_{\lambda(u,v)}(\Le,0,0)(\VPu,\Pu)^t\in\Img\Que \mathcal{L}(\lambda^*)\Qud$ implies, see (24), that
$$
\int_\Omega \varphi_1^2=0,
$$
a contradiction. 
%Thus
%\begin{align}\label{ItP1}
%	\mathcal{F}_{\lambda(u,v)}%(a(0)\Lu,0,0)\notin\Img\Que\mathcal{F}_{(u,v)}\Pae a(0)\Lu,0,0\Pad\Qud.
%\end{align}
Therefore, from Theorem 1.7 of [3] we conclude that 
%of \eqref{IuP1}, \eqref{IdP1} and \eqref{ItP1}, 
there exist $ \varepsilon>0 $ and applications of class $ C^1 $
\begin{align}
	\nonumber\lambda:&(-\varepsilon,\varepsilon)\Vai\Rr \\
	&\hspace{0.5cm}s\hspace{0.5cm}\Chega\lambda(s)=a(0)\Lu+\rho(s)	
\end{align}
and
\begin{align}
	\nonumber\Psi:&(-\varepsilon,\varepsilon)\Vai Z^2 \\
	&\hspace{0.5cm}s\hspace{0.5cm}\Chega\Psi(s)=(u(s),v(s))=\Pae s(\VPu+\varphi(s)),s(\Pu+\psi(s))\Pad , 
\end{align}
with $ \rho(0)=0 $ and $ \varphi(0)=\psi(0)=0 $. Furthermore, these are the only non-trivial solutions of (P) in a neighborhood of $ (a(0)\Lu,0,0) $.
\end{proof}

    In the next result\textcolor{green}{,} we study the direction bifurcation from the trivial solution.

\begin{Theo}
Assume that $a(0)>0$. Then, the bifurcation direction from the trivial solution $(u,v)=(0,0)$ at $\lambda=\lambda_1 a(0)$ is:
	\begin{itemize}
		\item[(a)] Supercritical, when 
  $$ a'(0)>\dfrac{\Pae b\rho-\lambda_1-\sigma\Pad\IntO\VPu^3}{\Lu \rho\IntO\VPu}. 
  $$
		\item[(b)] Subcritical, when 
  $$ a'(0)<\dfrac{\Pae b\rho-\lambda_1-\sigma\Pad\IntO\VPu^3}{\Lu \rho\IntO\VPu}. 
  $$ 
	\end{itemize}
\end{Theo}
\begin{proof}
	Using the equations (26) and (27) in the first equation of (P), it follows that
	\begin{align*}
		-a\Pae\IntO s\Pae\Pu+\psi(s)\Pad\Pad\Lapla\Pae s\Pae\VPu+\varphi(s)\Pad\Pad
		=&\Pae\Lu a(0)+\rho(s)\Pad\Pae s\Pae\VPu+\varphi(s)\Pad\Pad \\
		&-\Pae s\Pae\VPu+\varphi(s)\Pad\Pad^2 \\
		&+b\Pae s\Pae\VPu+\varphi(s)\Pad\Pad\Pae s\Pae\Pu+\psi(s)\Pad\Pad.
	\end{align*}
	Consider the Taylor series of $a$ centered at $0$, that is $a(s)=a\Pae 0\Pad+s a'\Pae 0\Pad+O(s)$, and \break $ \rho(s)=s\rho_1(s)+O(s) $. Note that
	\begin{align*}
		-\Pae a(0)+sa'(0)\IntO\Pae\Pu+\psi(s)\Pad+O(s)\Pad\Lapla\Pae\VPu+\varphi(s)\Pad s=\\		
        &\hspace{-3.6cm}=\Pae\Lu a(0)+s\rho_1(s)+O(s)\Pad\Pae\VPu+\varphi(s)\Pad s-\Pae\VPu+\varphi(s)\Pad^2 s^2\\
		%&\hspace{-.6cm}-\Pae\Pu+\psi(s)\Pad s^2\\
		&\hspace{-3.3cm}+b\Pae\VPu+\varphi(s)\Pad\Pae\Pu+\psi(s)\Pad s^2
	\end{align*}
	and, consequently,
	\begin{align*}
		-a(0)\Lapla\Pae\VPu+\varphi(s)\Pad s-s^2a'(0)\Lapla\Pae\VPu+\varphi(s)\Pad\IntO\Pae\Pu+\psi(s)\Pad-O(s)\Lapla\Pae\VPu+\varphi(s)\Pad s=\\	
        &\hspace{-11.5cm}=\Lu a(0)\Pae\VPu+\varphi(s)\Pad s+\rho_1(s)\Pae\VPu+\varphi(s)\Pad s^2+O(s)\Pae\VPu+\varphi(s)\Pad s-s^2\Pae\VPu+\varphi(s) \Pad^2 \\
		&\hspace{-11.2cm}+b\VPu\Pae\Pu+\psi(s)\Pad s^2+b\varphi(s)\Pae\Pu+\psi(s)\Pad s^2\textcolor{green}{.}
	\end{align*}
	Taking the terms of first order in $s$, it follows that 
	\begin{align*}
		-a(0)\Lapla\Pae\VPu+\varphi(s)\Pad-O(s)\Lapla\Pae\VPu+\varphi(s)\Pad
        =\Lu a(0)\Pae\VPu+\varphi(s)\Pad+O(s)\Pae\VPu+\varphi(s)\Pad
	\end{align*}
	and, consequently,
	\begin{align*}
		-a(0)\Lapla\VPu=\Lu a(0)\VPu\textcolor{green}{.}
	\end{align*}
	Let us now consider the  terms of second order in $s^2$, it follows that 	
	\begin{align*}
		-a'(0)\Lapla\Pae\VPu+\varphi(s)\Pad\IntO\Pae\Pu+\psi(s)\Pad
		=&\rho_1(s)\Pae\VPu+\varphi(s)\Pad-\Pae\VPu+\varphi(s) \Pad^2+b\VPu\Pae\Pu+\psi(s)\Pad\\
		&+b\varphi(s)\Pae\Pu+\psi(s)\Pad \textcolor{green}{.}
	\end{align*}
	Multiplying by $\VPu$ an integrating in $ \Omega $, it follows that
	\begin{align*}
		-a'(0)\IntO\Pae\Pu+\psi(s)\Pad\IntO\VPu\Lapla\Pae\VPu+\varphi(s)\Pad
		=&\rho_1(s)\IntO\VPu\Pae\VPu+\varphi(s)\Pad-\IntO\VPu\Pae\VPu+\varphi(s) \Pad^2\\
		&+b\IntO\VPu^2\Pae\Pu+\psi(s)\Pad+b\IntO\VPu\varphi(s)\Pae\Pu+\psi(s)\Pad
	\end{align*}
	and, by the  Green's Identity and the definition of eigenvalue, 
	\begin{align*}
		a'(0)\Lu\IntO\Pae\Pu+\psi(s)\Pad\IntO\VPu\Pae\VPu+\varphi(s)\Pad
		=&\rho_1(s)\IntO\VPu\Pae\VPu+\varphi(s)\Pad-\IntO\VPu\Pae\VPu+\varphi(s) \Pad^2\\
		&+b\IntO\VPu^2\Pae\Pu+\psi(s)\Pad+b\IntO\VPu\varphi(s)\Pae\Pu+\psi(s)\Pad\textcolor{green}{.}
	\end{align*}
       Using (22),  we get
	\begin{align*}
		\lim_{s\to 0}\rho_1(s)
		&=\dfrac{a'(0)\Lu K\IntO\VPu\IntO\VPu^2+\IntO\VPu^3-bK\IntO\VPu^3}{\IntO\VPu^2}\\
            &=a'(0)\Lu K\IntO\VPu+\Pae 1-bK\Pad\IntO\VPu^3.
	\end{align*}
We conclude the result from (22).
\end{proof}
%%%%%%%%%%%%%%%%%%%%%%%%%%%%%%%%%%%%%%%%%%%%%%%%%%

%%%%%%%%%%%%%%%%%%%%%%%%%%%%%%%%%%%%%%%%%%%%%%%%%%
\section{Global Bifurcation Analysis}

Let $U=C_0^1\FeOmega$ and the positive cone 
\begin{align*}
   \mathcal{P}=\{u\in U:\mbox{$u(x)\geq 0$ for all $x\in\Omega$}\}. 
\end{align*}
In this section, we will re-write (P) to show the existence of an unbounded  continuum \break $\Coc\subset\Rr\times\PC\times\PC$ of positive solutions of (P).\\

 In the next result, we prove a existence of a continuum $\Coc$ of positive solutions of (P), that is, a maximal connected and closed set in the set of positive solutions of (P).

\begin{Theo}\label{MRp1}
Assume that $a(0)>0$. From the point
    \begin{align}
        \Pae\lambda,u,v\Pad=\Pae a(0)\lambda_1,0,0\Pad,
    \end{align}
    bifurcates an unbounded continuum $\Coc\subset\Rr\times\PC\times\PC$ of coexistence states of (P).

Moreover;
\begin{enumerate}
\item  If $b\leq 0$, then
$$
(a(0)\lambda_1,\infty)\subset \Proj_\Rr(\Coc)\subset (\MinA\lambda_1,\infty).
$$
\item If $b>0$ and (H$_1$) is verified, then 
$$
(a(0)\lambda_1,\infty)\subset \Proj_\Rr(\Coc)\subset (0,\infty).
$$
\item If $b>0$ and $a$ verifies (H$_2$), there exists $\lambda_0\leq 0$ such that 
$$
(a(0)\lambda_1,\infty)\subset \Proj_\Rr(\Coc)\subset (\lambda_0,\infty).
$$
\end{enumerate}
    
\end{Theo}
\begin{proof}
First, from the result of Section 4, we conclude that $\Le=a(0)\Lu$ is a bifurcation point from the trivial solution $(0,0)$. 

On the other hand, note that it is possible rewrite the equations of (P) as follows
$$
\Gg(\lambda,u,v)=0
$$
where $\Gg:\Rr\times U\times U \mapsto U\times U$ is defined by
\begin{align*}
        \Gg\Pae\lambda,u,v\Pad=\begin{bmatrix}
		u  \\
		v				
	\end{bmatrix}-
         L\begin{bmatrix}
		\dfrac{\lambda u}{a(0)}  \\
		\rho u				
    \end{bmatrix}+N\Pae\lambda,u,v\Pad,   
    \end{align*}
where 
$$
L= \begin{bmatrix}
            \Pae-\Lapla\Pad^{-1} & 0\\
            0  & \Pae-\Lapla+\sigma\Pad^{-1}
        \end{bmatrix},
$$
and    
$$
N\Pae\lambda,u,v\Pad=L\begin{bmatrix}
		\lambda u\Pae\dfrac{1}{a\Pae\IntO v\Pad}-\dfrac{1}{a(0)}\Pad+\dfrac{buv-u^2}{a\Pae\IntO v \Pad}  \\
		0				
    \end{bmatrix}.
$$ 
Observe that 
\begin{align*}
   \dfrac{\|N\Pae\lambda,u,v\Pad\|_{_{\scriptscriptstyle U\times U}}}{\Noe\Pae u,v\Pad\Nod_{_{\scriptscriptstyle U\times U}}}\to 0\quad\mbox{as $\Noe\Pae u,v\Pad\Nod_{_{\scriptscriptstyle U\times U}}\to 0$.} 
\end{align*}
Now, we can apply Theorem 6.4.3 of [2], and following exactly the lines of Theorem 1.1 of [9] (see also Theorem 4.1 of [2]) and  conclude that there exists a continuum $\Coc$ of coexistence states of (P) emanating from $\Pae\Le,0,0\Pad$ which satisfies at least one of the following alternatives:
    \begin{itemize}
        \item[(A$_1$)] $\Coc$ is unbounded in $\Rr\times U\times U$.
        \item[(A$_2$)] There exists $\Pae\lambda_n,u_n,v_n\Pad\in\Coc$ such that $\lambda_n\to\Lb$ and $\Pae u_n,v_n\Pad\to\Pae 0,0\Pad$, with $\Lb\neq\Le$.
        \item[(A$_3$)] There exists $\Pae\lambda_n,u_n,v_n\Pad\in\Coc$ such that $\lambda_n\to\Lb$ and $\Pae u_n,v_n\Pad\to\Pae \UB,0\Pad$, with $\Lb\neq\Le$ and $\UB>0$.  
        \item[(A$_4$)] There exists $\Pae\lambda_n,u_n,v_n\Pad\in\Coc$ such that $\lambda_n\to\Lb$ and $\Pae u_n,v_n\Pad\to\Pae 0,\VB\Pad$, with $\Lb\neq\Le$ and $\VB>0$.
    \end{itemize}
    Now, we check the validity of each alternative and prove that only (A$_1$) holds.\\
    Alternative (A$_2$) is not valid: Suppose there exists $\Pae\lambda_n,u_n,v_n\Pad\in\overline{\Coc}$ such that $\lambda_n\to\Lb$ and $\Pae u_n,v_n\Pad\to\Pae 0,0\Pad$ in $U\times U$, with $\Lb\neq\Le$. Observe that
    $$
    \lambda_n=\sigma_1\left[-a\left(\int_\Omega v_n\right)\Delta+u_n-bv_n\right].
    $$
    Since $\Pae u_n,v_n\Pad\to (0,0)$ in $U\times U$, we conclude that
    $$
   \lambda_n= \sigma_1\left[-a\left(\int_\Omega v_n\right)\Delta+u_n-bv_n\right]\to a(0)\lambda_1,
    $$
    a contradiction.
    
  %  Consider     
   % \begin{align*}
    %    w_n:=\dfrac{u_n}{\Noe u_n\NodI}
    %\end{align*}
    %and note that $w_n\to w$ in $\Cd\FeOmega$, with $w>0$. %Taking $\lambda_n$, $u_n$ and $v_n$ in the first equation %of \eqref{p1}
    %\begin{align*}
    %    -a\Pae\IntO v_n\Pad\Lapla u_n=\lambda_nu_n-u^2_n+bu_nv_n
    %\end{align*}
    %and, dividing the equation above by $\Noe u_n\NodI$ and %taking the limit,
    %\begin{align*}
    %\begin{cases}\begin{array}{rll}
%				-a(0)\Lapla w=&\hspace*{-0.25cm}\Lb w,&\mbox{ in %}\Omega			\\
%				w=&\hspace*{-0.25cm}0, &\mbox{ on }\partial\Omega
%		\end{array}.\end{cases}
%	\end{align*}
 %   Thus $\Lb=a(0)\Lu$, which is a contradiction.\\
    \noindent Alternative (A$_3$) is false: 
    %\textcolor{red}{VOLTAR} 
    Suppose there exists $\Pae\lambda_n,u_n,v_n\Pad\in\overline{\Coc}$ such that $\lambda_n\to\Lb$ and $\Pae u_n,v_n\Pad\to\Pae \UB,0\Pad$ in  $U\times U$, with $\Lb\neq\Le$ and $\UB>0$. 
    %Consider (ESTE RAZONAMIENTO NO ME CONVENCE)
    %\begin{align*}
     %   z_n:=\dfrac{v_n}{\Noe v_n\NodI}
    %\end{align*}
    %and note that $z_n\to z$ in $C^2(\overline\Omega)$, with %$z>0$. (POR QUE). Taking $u_n$ and $v_n$ in the second %equation of \eqref{p1}
    %\begin{align*}
     %   -\Lapla v_n+\sigma v_n=\rho u_n
    %\end{align*}
    %and, dividing the equality above by $\Noe v_n\NodI$
    %\begin{align*}
    %    -\Lapla z_n+\sigma z_n=\rho\dfrac{u_n}{\Noe v_n\NodI}.
    %\end{align*}
    %Consider $\varphi\in\Ciz\FeOmegaZ$, with $\Omega_0$ a subset %of $\Omega$, and note that
    %\begin{align*}
     %   -\IntOZ\varphi\Lapla z_n+\sigma\IntOZ\varphi %z_n=\dfrac{\rho}{\Noe v_n\NodI}\IntOZ\varphi u_n.
    %\end{align*}
    %Thus, taking the limit on the equality above,
    %\begin{align*}
     %   -\IntOZ z\Lapla \varphi+\sigma\IntOZ\varphi z=\dfrac{\rho}%{\Noe v_n\NodI}\IntOZ\varphi u_n\to+\infty,
    %\end{align*}
    %which is a contradiction.\\
    %(OTRO RAZONAMIENTO)
Since $v_n$ is bounded in $C^{2,\gamma}(\overline\Omega)$, $\gamma\in (0,1)$, it is easy to deduce that $v_n\to v^*\geq 0$ solution  of 
$$
-\Delta v^*+\sigma v^*=\rho  \UB>0,
$$
a contradiction because $v^*=0$.
    
    \noindent Alternative (A$_4$) is not valid: Suppose there exists $\Pae\lambda_n,u_n,v_n\Pad\in\overline{\Coc}$ such that $\lambda_n\to\Lb$ and $\Pae u_n,v_n\Pad\to\Pae 0,\VB\Pad$ in $U\times U$, with $\Lb\neq\Le$ and $\VB>0$.  Using (10) we get
    %Taking $u_n$ and $v_n$ in the second equation for %\eqref{p1}
    %\begin{align*}
     %   -\Lapla v_n+\sigma v_n=\rho u_n
    %\end{align*}
    %and then
    \begin{align*}
        \sigma\Noe v_n\NodI\mI\rho\Noe u_n\NodI.
    \end{align*}
    Thus, taking limit,
    \begin{align*}
        \sigma\Noe\VB\NodI\mI 0,
    \end{align*}
    which is a contradiction.\\
    Therefore, from the point $\Pae a\Pae 0\Pad\Lu,0,0\Pad$ bifurcates an unbounded continuum $\Coc\subset\Rr\times\Cuz\FeOmega\times\Cuz\FeOmega$ of coexistence states of (P).

Now, assume that $b\leq 0$. then, by Proposition 3.1 we get that existence of a priori bounds in $L^\infty$, and by elliptic regularity, in $C^1\FeOmega$. Moreover, by  Proposition 3.1 we deduce non-existence of positive solutions for $\lambda\leq\MinA\lambda_1$. This finishes this case.

Assume $b>0$ and (H$_1$). By Proposition 3.2,  we conclude the existence of a priori bounds and non-existence for $\lambda\leq 0$.

Finally, assume $b>0$ and $a$ verifies $(H_2)$. Again, we obtain the existence of a priori bounds and non-existence of positive solutions for $\lambda\leq \lambda_0$ by Proposition 3.2.

This completes the proof.

\end{proof}

\section{Uniqueness result}
In this section we show a uniqueness result of coexistence states of (P). Specifically,
\begin{Theo}\label{uni}
Assume $a$ is increasing. Then, there exists $b_0>0$
such that (P) possesses at most a coexistence state for $b\in(-b_0,b_0)$.
\end{Theo}
\begin{proof}
First, we study the case $b=0$. Observe that in this case, multiplying the second equation of (P) by $e_\sigma$ and integrating in $\Omega$, we get that
$$
\int_\Omega v=\rho\int_\Omega e_\sigma u.
$$
Hence, in the case $b=0$, (P) is equivalent to  
\begin{align}\tag{$P_0$}
		\begin{cases}\begin{array}{rll}
				-a\Pae\rho \IntO e_\sigma u\Pad\Lapla u=&\hspace*{-0.25cm}\lambda u-u^2,&\mbox{ in }\Omega			\\
				-\Lapla v+\sigma v=&\hspace*{-0.25cm}\rho u, &\mbox{ in }\Omega			\\
				u=v=&\hspace*{-0.25cm}0, &\mbox{ on }\partial\Omega
		\end{array}
  \end{cases}
	\end{align}
Since $a$ is increasing, it follows by Theorem 5 in [5] the uniqueness of $u_0$ positive solution of
$$
-a\Pae\rho \IntO e_\sigma u\Pad\Lapla u=\lambda u-u^2,\quad\mbox{in $\Omega$,}\qquad u=0\quad \mbox{on $\partial\Omega$.}	
$$
Now, it is evident the uniqueness of a positive solution $v_0$ of 
the linear equation
$$
-\Lapla v+\sigma v=\rho u_0, \quad\mbox{ in $\Omega$,}\quad		
				v=0\quad\mbox{ on $\partial\Omega.$}
$$
Moreover, observe that
$$
\sigma_1\left[-a\Pae\rho \IntO e_\sigma u_0\Pad\Lapla+u_0-\lambda\right]=0, 
$$
and then, 
\begin{equation}
   \sigma_1\left[-a\Pae\rho \IntO e_\sigma u_0\Pad\Lapla+2u_0-\lambda\right]>0. 
\end{equation}

Next, we study the case $b\neq 0$. In this case we apply the Implicit Function Theorem. Consider the spaces $E$ and $F$ defined in (18) and define the map $\mathcal{G}:\Rr\times E\mapsto F$ by
\begin{align*}
	\mathcal{G}(b,u,v)=\begin{bmatrix}	
		-a\Pae\IntO v\Pad\Delta u-\lambda u+u^2-buv		\\
		-\Delta v+\sigma v-\rho u					
	\end{bmatrix}.
\end{align*}
%$$
%\mathcal{G}(b,u,v)=(-a(\int_\Omega v)\Delta u-\lambda u+u^2-buv,-\Delta v+\sigma v-\rho u).
%$$
Observe that $\mathcal {G}$ is derivable and $\mathcal{G}(0,u_0,v_0)=0$. Its derivative at $(0,u_0,v_0)$ is given by
\begin{align*}
	\mathcal{G}_{(u,v)}(0,u_0,v_0)\left[
\begin{array}{c}
\xi\\
\eta
\end{array}
\right]=\begin{bmatrix}	
		-a\Pae\IntO v_0\Pad\Delta\xi+M(x)\IntO\eta+N(x)\xi		\\
		-\Delta \eta+\sigma\eta-\rho\xi					
	\end{bmatrix},
\end{align*}
where 
$$
M(x)=\dfrac{a'\Pae\IntO v_0\Pad}{a\Pae\IntO v_0\Pad}(\lambda u_0-u_0^2),\qquad N(x)=2u_0-\lambda.
$$
%$$
%\mathcal{G}_{(u,v)}(0,u_0,v_0)
%\left(
%\begin{array}{c}
%\xi\\
%\eta
%\end{array}
%\right)
%=
%\left(
%\begin{array}{c}
%-a(\int_\Omega v_0)\Delta\xi+\frac{a'(\int_\Omega v_0)}{a(\int_\Omega v_0)}(\lambda u_0-u_0^2)\int_\Omega\eta-(\lambda -2u_0)\xi\\
%-\Delta \eta+\sigma\eta-\rho\xi
%\end{array}
%\right)
%$$
Now we show that $\mathcal{G}_{(u,v)}(0,u_0,v_0)$ is an   isomorphism from $E$ onto $F$. Given $(f,g)\in F$, we have to show the existence and uniqueness of solution of the problem 
 $$
 \mathcal{G}_{(u,v)}(0,u_0,v_0)
\left[
\begin{array}{c}
\xi\\
\eta
\end{array}
\right]=\left[
\begin{array}{c}
f\\
g
\end{array}
\right],
$$
or equivalently
\begin{align}
	\begin{cases}\begin{array}{rll}
			-a\Pae\IntO v_0\Pad\Delta\xi+M(x)\IntO\eta +N(x)\xi=&\hspace*{-0.25cm}f, &\mbox{ in }\Omega			\\
			-\Delta \eta+\sigma\eta-\rho\xi=&\hspace*{-0.25cm}g, &\mbox{ in }\Omega			\\
			\eta=\xi=&\hspace*{-0.25cm}0, &\mbox{ on }\partial\Omega.
	\end{array}\end{cases}
\end{align}
%\begin{align}\label{nuevo1}
%		\begin{cases}\begin{array}{ll}
%				-a(\int_\Omega v_0)\Delta\xi+\frac{a'(\int_\Omega v_0)}{a(\int_\Omega v_0)}(\lambda u_0-u_0^2)\int_\Omega\eta-(\lambda -2u_0)\xi=f&\mbox{ in }\Omega			\\
%				-\Delta \eta+\sigma\eta-\rho\xi=g &\mbox{ in }\Omega			\\
%				\xi=\eta=0, &\mbox{ on }\partial\Omega
%		\end{array},\end{cases}
%	\end{align}
Observe that (30) is a linear system including a non-local term in $\eta$ in the $\xi$-equation. We transform this linear system into another linear one.  Indeed, multiplying the second equation by $e_\sigma$, solution of (5), we get that 
 $$
 \int_\Omega\eta=\rho\int_\Omega e_\sigma\xi +\int_\Omega e_\sigma g.
 $$
Then, (30) is equivalent to
\begin{align}
	\begin{cases}\begin{array}{rll}
			-a\Pae\rho\IntO e_\sigma u_0\Pad\Delta\xi+N(x)\xi+\rho M(x)\IntO  e_\sigma\xi=&\hspace*{-0.25cm}h, &\mbox{ in }\Omega			\\
			-\Delta \eta+\sigma\eta-\rho\xi=&\hspace*{-0.25cm}g, &\mbox{ in }\Omega			\\
			\eta=\xi=&\hspace*{-0.25cm}0, &\mbox{ on }\partial\Omega
	\end{array}\end{cases}
\end{align}
%\begin{align}\label{nuevo2}
%		\begin{cases}\begin{array}{ll}
%				-a(\rho\int_\Omega e u_0)\Delta\xi+N(x)\xi+\rhi M(x)\int_\Omega  e\xi=F&\mbox{ in }\Omega			\\
%				-\Delta \eta+\sigma\eta-\rho\xi=g &\mbox{ in }\Omega			\\
%				\xi=\eta=0, &\mbox{ on }\partial\Omega
%		\end{array},\end{cases}
%	\end{align}
 where
% $$
% n(x)=\rho\dfrac{a'\Pae\IntO v_0\Pad}{a\Pae\IntO %v_0\Pad}(\lambda u_0-u_0^2)\geq 0
% $$
% and
 $$
 h(x)=f(x)-M(x)\IntO e_\sigma g.
 $$
 Observe that from (29) it follows that
\begin{equation}
\Su\Que-a\Pae\rho\int_\Omega e_\sigma u_0\Pad\Delta+N(x)\Qud>0,
\end{equation}
and then, 
there exists a unique positive solution, denoted by $H$ of the linear equation
\begin{align*}
	\begin{cases}\begin{array}{rll}
			-a\Pae\rho\IntO e_\sigma u_0\Pad\Delta H+N(x)H=&\hspace{-0.25cm}e_\sigma, &\mbox{ in }\Omega\\
			H=&\hspace{-0.25cm}0,&\mbox{ on }\partial\Omega.
		\end{array}
  \end{cases}
\end{align*}
%$$
%-a(\rho\int_\Omega e u_0)\Delta E+(2u_0-\lambda)E=e\quad\mbox{in $\Omega$,}\qquad E=0\quad\mbox{on $\partial\Omega$.}
%$$
Multiplying this equation by $\xi$ and integrating in $\Omega$, we obtain
$$
\int_\Omega Hh-\rho\int_\Omega H M\int_\Omega e_\sigma\xi=\int_\Omega e_\sigma\xi.
$$
Then, 
$$
\int_\Omega e_\sigma\xi=\frac{\int_\Omega H h}{1+\rho\int_\Omega H M},
$$
where we have used that $M\geq 0$ in $\Omega$. Hence, the first equation of (31) is equivalent to
\begin{align}
	\begin{cases}\begin{array}{rll}
			-a\Pae\rho\IntO e_\sigma u_0\Pad\Delta\xi+N(x)\xi=&\hspace{-0.25cm}h-\rho M(x)\dfrac{\IntO H f}{1+\rho\IntO H M}, &\mbox{ in }\Omega\\
			\xi=&\hspace{-0.25cm}0,&\mbox{ on }\partial\Omega
		\end{array}.\end{cases}
\end{align}
%\begin{equation}
%\label{nuevo3}
%\left\{
%\begin{array}{ll}
%-a(\rho\int_\Omega e u_0)\Delta\xi+(2u_0-\lambda)\xi=F-n(x)\frac{\int_\Omega E F}{1+\int_\Omega E n} & \mbox{in $\Omega$,}\\
%\xi=0 & \mbox{on $\partial\Omega$.}
%\end{array}
%\right.
%\end{equation}
By (32), it follows the existence and uniqueness of $\xi$ solution of (33). Since $\Su\Que-\Delta+\sigma\Qud>0$, it follows  the existence and uniqueness of $\eta$. This completes the proof.
\end{proof}   
%%%%%%%%%%%%%%%%%%%%%%%%%%%%%%%%%%%%%%%%%%%%%%%%%%
\section{Case $a(0)=0$}
In this section, we study the case when $a(0)=0$. Observe that in this case we can not apply directly the results of previous sections. For this, it is necessary to consider the following auxiliary problem
    \begin{align}\tag{NP$_1$}
		\begin{cases}\begin{array}{rll}
				-a_\varepsilon\Pae\IntO v\Pad\Lapla u=&\hspace*{-0.25cm}\lambda u-u^2+buv,&\mbox{ in }\Omega			\\
				-\Lapla v+\sigma v=&\hspace*{-0.25cm}\rho u, &\mbox{ in }\Omega			\\
				u=v=&\hspace*{-0.25cm}0, &\mbox{ on }\partial\Omega
		\end{array}
  \end{cases}
	\end{align}
where 
$$
\An(s):=a(s)+\varepsilon,\quad\mbox{for $\varepsilon>0$.} 
$$
\begin{Theo}
\label{teocero}
Assume that $a(0)=0$ and $b\leq 0$ or $b>0$ and $(H_1)$ or $(H_2)$. Then, for $\lambda>0$ there exists at least a coexistence state of (P).
\end{Theo}
\begin{proof}
For each $\varepsilon>0$ there exists an unbounded continuum
$\mathcal{C}_\varepsilon$ of positive solutions of (NP$_1$) bifurcating from the trivial solution at
$$
\lambda_\varepsilon^*=a_\varepsilon(0)\lambda_1=\varepsilon \lambda_1,
$$
and 
$$
(\varepsilon \lambda_1,\infty)\subset \mbox{Proj}_\Rr(\mathcal{C}_\varepsilon).
$$
Hence, for $\lambda>0$ fixed, there exists $\varepsilon_0>0$ such that there exists at least a coexistence state $(u_\varepsilon,v_\varepsilon)$ of (P) for $\varepsilon<\varepsilon_0$. Observe that:
    $$
    -a\Pae\IntO v_\varepsilon\Pad\Delta u_\varepsilon\geq u_\varepsilon(\lambda+(bv_\varepsilon)_{_{\scriptscriptstyle L}}-u_\varepsilon).
    $$
Thus, $u_\varepsilon$ is supersolution of a logistic equation of the form
$$
-d\Delta w=w(\mu-w)\quad\mbox{in $\Omega$,}\quad w=0\quad \mbox{on $\partial\Omega$.}
$$
It is well known that
$$
\frac{\mu-d\lambda_1}{\|\varphi_1\|_\infty}\varphi_1(x)\leq w,
$$
where $\varphi_1$ is a positive eigenfunction associated to $\lambda_1$. Then,  
\begin{equation}
\dfrac{\lambda+(bv_\varepsilon)_{_{\scriptscriptstyle L}}-a\Pae\IntO v_\varepsilon\Pad\lambda_1}{\|\varphi_1\|_\infty}\varphi_1(x)\leq u_\varepsilon.
\end{equation}
We claim that
\begin{equation}
    \|u_\varepsilon\|_\infty\leq C
\end{equation}
for some positive constant independent of $\varepsilon$.

Indeed, assume first that $b\leq 0$. In this case, by Proposition 3.1 we get that $\|u_\varepsilon\|_\infty\leq \lambda$.

In the case $b>0$ and $(H_1)$, we obtain from Proposition 3.2 that
$$
\|u_\varepsilon\|_\infty\leq C,
$$
with $C$ a positive constant independent of $\varepsilon$.

Finally, when $b>0$ and $(H_2)$, we can follow again the proof of Proposition 3.2 and conclude that 
$$
\|u_\varepsilon\|_\infty\leq C,
$$
with $C$ a positive constant independent of $\varepsilon$.

Hence, in all the cases we have proved the claim (35).

Going back to the equation of $v_\varepsilon$, we obtain that $v_\varepsilon$ is bounded in $W^{2,p}(\Omega)$, for any $p>1$, and then passing to the limit
$$
v_\varepsilon\to v^*\geq 0\quad\mbox{in $C^1\FeOmega$}.
$$
Assume that  $v^*=0$. Then, by (34)
$$
u_\varepsilon(x)\geq \dfrac{\lambda}{2\|\varphi_1\|_\infty}\varphi_1(x).
$$
Then,
$$
-\Delta v_\varepsilon+\sigma v_\varepsilon=\rho u_\varepsilon\geq C\varphi_1(x),
$$
and then $v^*(x)>0$ for all $x\in\Omega$. As consequence, $v^*>0$ in $\Omega$, and then 
$$
a\Pae\IntO v_\varepsilon\Pad\to a\Pae\IntO v^*\Pad>0.
$$
Hence, $u_\varepsilon$ is bounded in  $W^{2,p}(\Omega)$, and then 
$$
u_\varepsilon\to u^*> 0\quad\mbox{in $C^1\FeOmega$}.
$$
It is clear that $(u^*,v^*)$ is a coexistence state of (P). This finishes the proof.
\end{proof}
\noindent{\bf Acknowledgments:} Marcos A. Viana Costa and Yino Cueva Carranza has been partially supported by PROPG/Unesp.

\bibliographystyle{abbrv}  
%\bibliography{references} 
%%%%%%%%%%%%%%%%%%%%%%%%%%%%%%%%%%%%%%%%%%%%%%%%%%
%%%%%%%%%%%%%%%%%%%%%%%%%%%%%%%%%%%%%%%%%%%%%%%%%%
%%%%%%%%%%%%%%%%%%%%%%%%%%%%%%%%%%%%%%%%%%%%%%%%%%

%%%%%%%%%%%%%%%%%%%%%%%%%%%%%%%%%%%%%%%%%%%%%%%%%%
%%%%%%%%%%%%%%%%%%%%%%%%%%%%%%%%%%%%%%%%%%%%%%%%%%
%%%%%%%%%%%%%%%%%%%%%%%%%%%%%%%%%%%%%%%%%%%%%%%%%%
\end{document}